\documentclass[a4paper,12pt]{amsart}
\usepackage{enumerate}

\usepackage[latin1]{inputenc}
\usepackage[T1]{fontenc}
\usepackage[english]{babel}

\usepackage{mathrsfs}
\usepackage{amscd}
\usepackage{amsfonts}
\usepackage{amsmath}
\usepackage{amssymb}
\usepackage{amstext}
\usepackage{amsthm}
\usepackage{amsbsy}

\usepackage{xspace}
\usepackage[all]{xy}
\usepackage{graphicx}
\usepackage{url}
\usepackage{latexsym}

\newtheorem{theo}{Theorem}[section]

\newtheorem{claim}[theo]{Claim}

\theoremstyle{definition}
\newtheorem{dfn}[theo]{Definition}

\newtheorem{fact}[theo]{Fact}

\newcommand{\N}{\ensuremath{\mathbb{N}}}
\newcommand{\Z}{\ensuremath{\mathbb{Z}}}

\newcommand{\R}{\ensuremath{\mathbb{R}}}

\newcommand{\A}{\ensuremath{\aleph}}

\newcommand{\M}{\ensuremath{\mathcal{M}}}

\newcommand{\vs}{\vspace{0.3cm}}

\DeclareMathOperator{\rk}{rk}

\newcommand{\divhull}[1]{\langle {#1}\rangle_{\mathbb{Q}}}
\newcommand{\rc}[1]{RC(\mathbb{Q}({#1}))}

\newcommand{\book}[2]{{\scshape#1}, {\bf #2}}
\newcommand{\pre}[2]{{\scshape#1}, #2}
\newcommand{\publ}[6]
{{\scshape#1}, #2, {\itshape #3}, {\bf #4} (#5), pp.~#6.}

\title[Recursively saturated real closed fields]{A valuation theoretic characterization of recursively saturated real closed fields}

\author{Paola D'Aquino}
\address{Dipartimento di Matematica, Seconda Universit\`a di Napoli, Italy}
\email{paola.daquino@unina2.it}
\author{Salma Kuhlmann}
\address{FB Mathematik \& Statistik, Universit\"at Konstanz,
Germany}
 \email{salma.kuhlmann@uni-konstanz.de}
  \author{Karen Lange}
\address{Department of Mathematics, Wellesley College, United States}
 \email{karen.lange@wellesley.edu}
\date{\today}

\begin{document}
\begin{abstract}
 We  give a valuation theoretic characterization for a real closed field to be recursively saturated.  This builds on work in
in \cite{kkmz}, where the authors gave such a characterization for  $\kappa$-saturation, for a cardinal $\kappa
\geq \aleph_0$.  Our result extends the characterization of Harnik and Ressayre \cite{hr} for a divisible ordered abelian group to be recursively saturated.
\end{abstract}

\maketitle


\vs
\section {Introduction}

%


Recursive saturation was introduced by Barwise and Schlipf in \cite{bs}.
\begin{dfn}(\cite{bs})
A structure \M\ for a computable language $L$ is {\em recursively saturated} if given a computable set of $L$-formulas $\tau(x, \bar{v})$ and  a tuple $\bar{a}$ in \M\ appropriate to substitute for $\bar{v}$ such that every finite subset of $\tau(x,\bar{a})$
is satisfied in \M, then the whole $\tau(x,\bar{a})$ is satisfied in \M.
\end{dfn}
\noindent In \cite{dks} a characterization of countable recursively
saturated real closed fields was obtained in terms of their integer
parts. $\kappa$- saturation (for an arbitrary infinite cardinal
$\kappa$) has been investigated in terms of valuation theory for
divisible ordered abelian groups in \cite{sgr}, real closed fields
in \cite{kkmz}, and more generally for o-minimal expansions of real
closed fields in \cite{cdk}. In this paper we extend the above
valuation theoretical characterizations to recursively saturated
real closed fields, thereby extending results of \cite{hr} for
divisible ordered abelian groups.


\section{Preliminaries}
\subsection{Scott sets}

A subset $\mathcal{T}\subset 2^{<\omega}$ is a {\em tree} if every substring of an element of $\mathcal{T}$ is also an element of $\mathcal{T}$.  If $\sigma, \tau\in 2^{<\omega}$, we let $\sigma\prec\tau$ denote that $\sigma$ is a substring of $\tau$.   A sequence $f\in2^{\omega}$ is a {\em path} through a tree $\mathcal{T}$ if for all $\sigma\in 2^{<\omega}$ with $\sigma\prec f$, the string $\sigma$ is an element of $\mathcal{T}$.
 For any string $\sigma\in 2^{<\omega}$, the length of $\sigma$, denoted by $length(\sigma)$, is the unique  $n\in\omega$ satisfying $\sigma\in 2^n$.
\begin{dfn}
 A nonempty set $S\subset \mathbb{R}$ is a {\em Scott set} if
\begin{enumerate}
\item $S$ is computably closed, i.e., if $r_1,\ldots r_n\in S$ and $r\in\mathbb{R}$ is computable from (the Turing join of) $r_1,\ldots r_n$, then $r\in S$.
\item If an infinite tree $\mathcal{T}\subset 2^{<\omega}$ is computable in some $r\in S$, then $\mathcal{T}$ has a path that is computable in some $r'\in S$.
\end{enumerate}
\end{dfn}

\begin{fact}
Any Scott set $S$ is an archimedean real closed field.
\end{fact}

\subsection{Some valuation theoretic notions}
\label{DOAG}

We summarize the required background (see \cite{book} and \cite{sgr}). Let  $(G, +, 0, <)$ be a divisible  ordered  abelian group.  Given $A\subset G$, we let $\divhull{A}$ denote the smallest divisible ordered subgroup of $G$ containing $A$.   For any $x\in G$ let $|x|=\max \{x,-x\}$.  For non-zero $x, y \in G$ we define $x\sim y$ if there exists $n \in \N$ such that $n|x| \geq |y| $ and $ n|y| \geq |x|. $ We write $x<<y$ if $n|x| < |y|$ for all $n \in \N$. Clearly,  $\sim$ is an equivalence relation, and we let $[x]$ denote the equivalence class of any non-zero $x\in G$.
Let $\Gamma := G-\{ 0\}/\sim = \{[x] : x \in G-\{ 0\} \}$. We can define an order on $<_{\Gamma}$ in terms of $<<$  as follows, $[y]\, <_{\Gamma} [x] $ if  $x <<y$ (notice the reversed order).    Given a linear ordering $(A, <)$ and $A_1, A_2\subset A$, we use the notation $A_1<A_2$ to indicate that $a_1<a_2$ for all  $a_1\in A_1$ and $a_2\in A_2$.
\begin{fact}
{\it (a)} $\Gamma$ is a totally ordered set under $<_{\Gamma}$, and we will refer to it as the value set of $G$.

\noindent
{\it (b)} The map
\begin{align*}
v \colon &G\  \longrightarrow\ \Gamma \cup \{\infty\} \\
&0\quad \mapsto\quad \infty \\
&x\quad \mapsto\quad [x] \quad (\mbox{if }x \neq 0)
\end{align*}
is a valuation on $G$ as a $\Z$-module, i.e. for every $x, y \in G$: \newline $v(x) = \infty$  if and only if  $x = 0$,  $v(nx) = v(x)$  for all $n \in \Z$, $n \neq 0$, and $v(x+y) \geq \min\{v(x), v(y)\}$.

\vs
\item[$(c)$] For every $\gamma \in \Gamma$  the  Archimedean component associated to $\gamma$ is the maximal Archimedean subgroup of $G$ containing some $x\in \gamma$. We denote it by $A_{\gamma}$.  For each $\gamma$, $A_{\gamma}$ is isomorphic to an ordered subgroup of $(\mathbb R, +,0,<).$
Furthermore, we can calculate the isomorphism type of $A_\gamma$ in terms of any $x\in\gamma$.  Given \mbox{$x, y\in \gamma$,} we let $\frac{y}{x}=\sup\{r\in\mathbb{Q}\mid rx<y\}$, and let \mbox{$A_{\gamma, x}=\{\frac{y}{x}\mid y\in\gamma\}\cup\{0\}$.}   Then,  $A_\gamma\cong A_{\gamma, x}.$

\vs
\item[$(d)$]  Since $G$ is a divisible abelian group, we may view $G$ as a vector space over $\mathbb{Q}$. We focus on the case where $G$ is finite dimensional as a vector space, as we will use the next notion of valuation independence exclusively in that context.   A set $\{g_1, \ldots, g_n\}\subset G$ is called {\em valuation independent} if for all $q_1, \ldots q_n\in \mathbb{Q}$,
\begin{equation*}
v(q_1g_1+\ldots+q_n g_n)=\min\{ v(g_i)\mid q_i\not=0\}
\end{equation*}
A basis $\{g_1, \ldots, g_n\}$ for $G$ is called a {\em valuation basis} if it is a valuation independent set.
A theorem by Brown \cite{br} states that every vector space  of countable dimension with a valuation admits a valuation basis.
\end{fact}

\begin{dfn}
Let $\lambda$ be an infinite ordinal. A sequence $(a_{\rho})_{ \rho < \lambda}$  contained in $G$  is said to be {\it pseudo Cauchy} (or {\it pseudo convergent}) if for every
$\rho < \sigma < \tau<\lambda$ we have
\[
v(a_{\sigma} - a_{\rho})\ <\ v(a_{\tau} - a_{\sigma}).
\]

\end{dfn}

\begin{fact}
If $(a_{\rho})_{\rho<\lambda}$ is pseudo Cauchy sequence then for all $\rho < \sigma<\lambda$ we have
\[
v(a_{\sigma} - a_{\rho}) = v(a_{\rho + 1} - a_{\rho}).
\]
\end{fact}

\begin{dfn}
Let $(a_{\rho})_{\rho < \lambda}$  be a pseudo Cauchy sequence in $G$. We say that $x \in G$ is a
{\it pseudo limit} of $S$ if
\[
v(x - a_{\rho}) = v(a_{\sigma} - a_{\rho}) = v(a_{\rho + 1} - a_{\rho})  \quad \mbox{ for all } \rho < \sigma<\lambda.
\]

\end{dfn}

If $(R,+,\cdot ,0,1,<)$ is an ordered field then it has a natural valuation $v$, that is the natural valuation associated with the ordered abelian  group $(R,+ ,0,<~)$. We will denote by $G$ the value group of $R$ with respect to $v$, i.e., $G=v(R)$. If $(R,+,\cdot ,0,1,<)$ is a real closed field then $G$ is  divisible, and we will refer to the  linear dimension of $G$ as a $\mathbb Q$-vector space as the {\em rational rank} of $G$, denoted $\rk (G)$.
 For the natural valuation on $R$, we use the notations \mbox{$\mathcal O_R=\{ r\in R:v(r)\geq 0\}$} and \mbox{$\mu_R =\{ r\in R: v(r)>0\}$} for the valuation ring and the valuation ideal, respectively. The residue field $k$ is the quotient $\mathcal O_R/\mu_R$, and we recall that it is isomorphic to a  unique subfield of $\mathbb R$.  When convenient, we identify $k$ with this unique subfield of $\mathbb{R}$.  Given any $a\in R$ with $v(a)\ge 0$, we denote the residue of $a$ by $\overline{a}\in k$, i.e., $\overline{a}$  is the unique element in $k$ such that $v(a-\overline{a})>0$.
Notice that in the case of ordered fields there is a unique archimedean component up to isomorphism, and if the field is real closed, the archimedean component is the residue field.
\par
 If $R$ is a real closed field, given $X\subset R$, we let $\rc{X}$ denote the real closure of $\mathbb{Q}(X)$ in $R$.
A notion of pseudo Cauchy sequence is easily extended to any ordered field as in the case of ordered abelian groups.

\section{Background on $\kappa$-saturated structures}

We now recall the characterization of $\A_{\alpha}$-saturation for divisible ordered abelian groups given in  \cite{sgr}.  We need the notion of $\eta_\alpha$-sets (see \cite{rosenstein}).   An $\eta_\alpha$-set is a linear ordering $(A, <)$ such that, whenever $A_1, A_2\subset A$ have cardinality less than $\A_\alpha$ and  $A_1<A_2$, then there is an $a\in A$ such that $A_1<a<A_2$.  Observe that an $\eta_0$-set is simply a dense linear ordering without endpoints.

\begin{theo} \cite{sgr}
\label{doag}
Let $G$ be a divisible ordered abelian group, and let $\A_{\alpha}\geq \A_0$. Then $G$ is $\A_{\alpha}$-saturated in the language of ordered groups if and only

\begin{enumerate}[(i)]

\item
the value set of $G$ is an $\eta_{\alpha}$-set,

\item
all the archimedean components of $G$ are isomorphic to $\R$, and

\item
every pseudo Cauchy sequence in a divisible subgroup of $G$ with a value set of cardinality less than $\A_{\alpha}$ has a pseudo limit in $G$.

\end{enumerate}
\end{theo}

Notice that in  the case of $\A_0$-saturation  the necessary and sufficient conditions reduce only to (1) and (2).

\par

The following characterization of $\A_{\alpha}$-saturated real closed fields was obtained in \cite{kkmz}.

\begin{theo} \cite[Theorem 6.2]{kkmz}  \label{rcf}
Let $R$ be a real closed field, $v$ its natural valuation, $G$ its value group and $k$ its residue field.
Let $\A_{\alpha} \geq \A_0$. Then $R$ is $\A_{\alpha}$-saturated in the language of ordered fields if and only if

\begin{enumerate}[(i)]

\item $G$ is $\A_{\alpha}$-saturated,

\item $k \cong \R$,

\item every pseudo Cauchy sequence in a subfield of $R$ of absolute transcendence degree
less than $\A_{\alpha}$ has a pseudo limit in $R$.
\end{enumerate}
\end{theo}

 In the proof of Theorem \ref{rcf} the {\it dimension inequality} (see \cite{prestel}) is crucially used in the case of  $\A_0$-saturation. This says that the rational rank of the value group of a finite transcendental extension of a real closed field is bounded by the transcendence degree of the extension.

\section{Recursively saturated divisible ordered abelian groups}
Harnik and Ressayre  state the following result in \cite{hr} and
sketch a proof just for the necessity of condition (ii). We include
here a complete proof.
\begin{theo}\label{recsatG}
Let $G$ be a divisible ordered abelian group. Then $G$ is recursively saturated in the language of ordered groups  if and only if
\begin{enumerate}[(i)]
\item\label{GDLO} the value set of $G$ is  an $\eta_{0}$-set,  and
\item\label{GSS}  all  archimedean components of $G$ equal a common Scott set $S$.
\end{enumerate}
\end{theo}

\begin{proof}
 Suppose ${G}$ is recursively saturated.  We show that (\ref{GDLO}) and (\ref{GSS}) hold.

 \noindent
 (\ref{GDLO})   Let $g, g'\in G$ such that $g, g'>0$ and  $v(g)< v(g')$.  The partial type
 \begin{equation*}
 \beta(x, g, g')=\{ng'<x \mid n\in\mathbb{N}\}\cup \{x <ng\mid n\in\mathbb{N}\}
 \end{equation*}
is computable and finitely satisfiable (since $v(g)< v(g')$).  By recursive saturation, there is some $h\in G$ such that  $\beta(h, g, g')$ holds in $G$, and $v(g)< v(h)< v(g')$.

\noindent
(\ref{GSS}) We first show that $A_{[g], g}=A_{[g'], g'}$ for all nonzero $g, g'\in G$.  Let $r\in A_{[g], g}$.  Let $\delta(x, y, g, g')$ be the partial type  consisting of all formulas with $q, q'\in\mathbb{Q}$ and $q<q'$ of the form
\begin{equation*}
qg<y< q' g \rightarrow {q}g'<x<{q}'g'.
\end{equation*}
Since $r\in A_{[g], g}$, there exists some $h\in G$ so that $\frac{h}{g}=r$.  The set of formulas $\delta(x, h, g, g')$ is computable and finitely satisfiable in ${G}$ since ${G}$ is divisible, so there is some $h'\in G$ so that  $\delta(h', h, g, g')$ holds in $G$.  Then $\frac{h'}{g'}=r$, so $r\in A_{[g'], g'}$.  We have that  $A_{[g], g}=A_{[g'], g'}$ by a symmetric argument.  Hence, it is well defined to   refer to $A_{[g], g}$ simply as $A_{[g]}$.

Let $g\in G$.  We show that $A_{[g]}$ is a Scott set.  Suppose \mbox{$r_1, \ldots, r_n\in A_{[g]}$} and $r\in \mathbb{R}$ is computable in $r_1, \ldots, r_n$ via some Turing reduction $\Psi$.  Take $g_i\in G$ such that $r_i=\frac{g_i}{g}$ for $1\le i\le n$.  For each $n$-tuple of pairs of rationals $(q_i< q'_i)_{i=1}^n$, each stage $s\in\mathbb{N}$, and another pair of rationals $\hat{p}<\hat{p}'$, compute whether $\Psi$, using only the knowledge that arbitrary input reals $\tilde{r}_1, \ldots, \tilde{r}_n$ satisfy  $q_i<\tilde{r}_i< q'_i$ for $1\le i\le n$, halts in $s$ steps and outputs a real between $p$ and $p'$.  If $\Psi$ halts in this situation, enumerate the  formula
\begin{equation*}
(q_1g<g_1< q'_1 g \ \wedge \ldots \wedge \ q_ng<g_n< q'_n) \rightarrow pg<x<p'g
\end{equation*}
into the partial type $\zeta(x, g_1, \ldots, g_n)$.  The partial type $\zeta(x, g_1, \ldots, g_n)$ is computably enumerable and finitely satisfiable since ${G}$ is divisible.  By recursive saturation, there is some $h\in G$ such that  $\zeta(h, g_1, \ldots, g_n)$ holds in $G$.  Since $\Psi$ in fact computes $r$ from $r_1, \ldots, r_n$, we have that $r=\frac{h}{g}\in A_{[g]}$, as desired.

Let  $\mathcal{T}\subset 2^{<\omega}$ be an infinite tree computable in some $r\in A_{[g]}$.  We show that  $\mathcal{T}$ has a path computable in some $r'\in A_{[g]}$.  Fix a computable function
\begin{eqnarray}\label{treefcn}
f: 2^{<\omega} & \longrightarrow & \{(a, b)\in\mathbb{Q}^2\mid a<b\}\\
\sigma & \longrightarrow & I_\sigma=(a_\sigma, b_\sigma)
\end{eqnarray}
satisfying the following properties for $\sigma, \tau\in2^{<\omega}$:
   \begin{enumerate}[(a)]
\item $b_\sigma-a_\sigma= 2^{-length(\sigma)}$,
\item $I_\sigma\cap I_\tau=\emptyset$ if $\sigma\not\prec \tau$ and $\sigma\not\succ\tau$, and
\item $I_\sigma\subset I_\tau$ if $\sigma\succ\tau$.
\end{enumerate}

Let $\mathcal{T}\subset 2^{<\omega}$ be a tree that is computable from some $r\in A_{[g]}$ via a Turing reduction $\Lambda$.  Let $\mathcal{T}(k)$ be the set of nodes in $\mathcal{T}$ of length $k\in\omega$, and let ${I_{\mathcal{T}(k)}=\cup_{\sigma\in \mathcal{T}(k)} \, I_\sigma}.$  Fix some nonzero $g\in G$. Since $r\in A_{[g]}$, there is some $h\in G$ such that $r=\frac{h}{g}$.

 For each  pair of rationals $q< q'$ and  $s, k\in\mathbb{N}$, compute whether Turing reduction $\Lambda$, using only the information that an arbitrary input real $\tilde{r}$ satisfies  $q<\tilde{r}< q'$, halts in $s$ steps and outputs a (finite) set of nodes $A\subset 2^k$.   If $\Lambda$ halts in this situation, enumerate the  formula
\begin{equation*}
qg<h< q' g  \rightarrow \frac{x}{g}\in \cup_{\sigma\in A}\, I_{\sigma}
\end{equation*}
into the partial type $\kappa(x, {h}, g)$.  Note that $\frac{x}{g}\in \cup_{\sigma\in A}\, I_{\sigma}$ can be expressed as a quantifier free formula in the language of divisible ordered groups.  The partial type $\kappa(x, {h}, g)$ is computably enumerable and finitely satisfiable since ${G}$ is divisible.  By recursive saturation, there is some $h'\in G$ such that $\kappa(h', h, g)$ holds in $G$.   Since $\Lambda$  computes $\mathcal{T}$ from $r$, for each $k\in\mathbb{N}$, there is some stage $s\in\mathbb{N}$ such that $\Lambda$ computes $\mathcal{T}(k)$ from $r$.  Hence, for each $k\in\mathbb{N}$, there is some $q, q'\in\mathbb{Q}$ with $r\in (q, q')$ such that the formula
\begin{equation*}
qg<h< q' g  \rightarrow \frac{x}{g}\in \cup_{\sigma\in \mathcal{T}(k)}\, I_{\sigma}
\end{equation*}
is in $\kappa(x, h, g)$.  Since $r=\frac{h}{g}$, the real  $r'=\frac{h'}{g}\in A_{[g]}$ is in  $\cup_{\sigma\in \mathcal{T}(k)} \, I_{\sigma}$ for all $k\in\mathbb{N}$.  Then, the set $\mathcal{P}=\{\sigma\in \mathcal{T} \mid r'\in I_{\sigma}\}$ is a path through $\mathcal{T}$.  Moreover, $\mathcal{P}$ is computable from $r'$ since the assignment function $f$ of nodes to intervals is a computable function and if $r'\in I_{\tau}$ for some $\tau\in\mathcal{T}$, then $r'$ is in  $I_{\tau0}$ or $I_{\tau1}$ and it is computable to determine which one.   We have finished showing that $A_{[g]}$ is a Scott set, and the necessity portion of the theorem.

Now, let $G$ be a divisible ordered abelian group.  We  show that if $G$ satisfies (\ref{GDLO}) and (\ref{GSS}), then $G$ is recursively saturated.  Let $S\subset\mathbb{R}$ be the Scott set such that $S=A_{[g], g}$ for all $g\in G$ by (\ref{GSS}).  The proof follows the structure of the proof of Theorem \ref{doag} for the case of  $\aleph_0$-saturation.  The proof differs, however, in that its most  interesting aspect  is finding (using $S$) a complete extension of the  given computable partial type  that is finitely satisfiable with the given parameters.

Let $\bar{g}=(g_1, \ldots, g_n)$ be an $n$-tuple from $G$.  Let $\tau'(x, \bar{y})$ be a computable partial type so that $\tau'(x, \bar{g})$ is finitely satisfiable in $G$.  We first extend $\tau'(x, \bar{y})$ to a complete type $\tau(x, \bar{y})$ so that $\tau(x, \bar{g})$ is also finitely satisfiable in $G$.  We define an intermediate extension $\tau''(x, \bar{y})$ of $\tau'(x, \bar{y})$ first.

 Set ${G}'=\divhull{\bar{g}}$. We may assume that $\bar{g}=(g_1, \ldots, g_n)$ is a valuation basis for $G'$. Otherwise, we could replace the parameters $\bar{g}$ by a valuation basis $\bar{g}'$ by  substituting every occurrence of $g_i$ in $\tau'(x, \bar{g})$ with its definition over   $\bar{g}'$ in an effective way.  Similarly, we may assume that $0<g_1<g_2<\ldots <g_n$.


Let $h_1, \ldots, h_l\in \{g_1, \ldots, g_n\}$ satisfy

   \begin{enumerate}[(a)]
   \item $0<  h_1<< h_2<< \ldots<< h_l$, and
\item  for each $g_i$ with $1\le i\le n$, there is exactly one $j_i$ with  $1\le j_i\le l$ such that  $g_i\in [h_{j_i}]$.
\end{enumerate}

Let $r_{j_i}=\frac{g_i}{h_{j_i}}$.  By assumption (\ref{GSS}), $r_{j_i}\in A_{[h_{j_i}]} =S$.  Since $S$ is a Scott set, there is some $r'\in S$ (e.g.,  $r'= r_{j_1}\oplus\ldots\oplus r_{j_n},$ the {\em join} of the $r_{j_i}$) that computes each of the $r_{j_i}$.

Let  $\tau''(x, \bar{g})$ be a partial type that  contains all formulas in $\tau'(x, \bar{g})$ as well as the  formulas described below (written in terms of the appropriate parameters in $\bar{g}$).
\begin{itemize}
\item[(a$^\prime$)] For all $i$ satisfying $1\le i< l$ and $n\in\omega$, include the formula  $0<  h_i \ \wedge\ nh_i< h_{i+1}$.
\item[(b$^\prime$)] For all $q\in \mathbb{Q}$ and  all $i$ satisfying $1\le i\le n$, if $q<r_{j_i}$, include the formula $qh_{j_i}< g_i$ in $\tau''(x, \bar{g})$.  Similarly, if $q>r_{j_i}$, include the  formula  $ g_i<qh_{j_i}$ in $\tau''(x, \bar{g})$.
\end{itemize}

\begin{claim} A formula holds of $\bar{g}$ in $G$ if and only if this formula is  in any extension of   $\tau''(x, \bar{g})$.   Moreover, $\tau''(x, \bar{g})$ is computable from $r'$.
\end{claim}
\begin{proof}
Since divisible ordered abelian groups admit quantifier elimination and $\tau'(x, \bar{g})$ is computable, it suffices to show we can deduce the order of  any two  terms in $\bar{g}$ from formulas in $\tau''(x, \bar{g})$ computably from $r'$.  Note that the formulas added to $\tau'(x, \bar{g})$ by condition (a$^\prime$) are computable from $r'$ since knowing the order of elements $h_1, \ldots, h_l$ is only finite information.    Consider a linear combination \mbox{$s_1g_1+\ldots +s_n g_n$} where $s_1, \ldots, s_n\in \mathbb{Z}$.  Ordering  any two  non-equal terms in $\bar{g}$ is the same as determining whether such a non-trivial linear combination  is positive or negative.  Suppose $i_k$ is the largest index in the linear combination for which $s_{i_k}\not=0$.  Let $h=h_{j_{i_k}}$.   To determine whether a nonzero term is positive or negative, we simply need to determine whether the sum of  all monomials in this term with non-zero $s_i$ and valuation $h$ is positive or negative.  Suppose $s_{i_1} g_{i_1}+\ldots +s_{i_k}g_{i_k}$ is this sum.  Since $\bar{g}$ is a valuation basis, this new linear combination is nonzero and  is positive if and only if $s_{i_1} r_{j_{i_1}}+\ldots +s_{i_k}r_{j_{i_k}}>0$ (See \cite{har}, Propositions 12 and 13).  Hence,  we can compute from $r'$ the ordering between any two terms in $\bar{g}$.
\end{proof}

 Since $\tau'(x, \bar{g})$ is finitely satisfiable in $G$,  the claim guarantees that  $\tau''(x, \bar{g})$ is finitely satisfiable in $G$ as well as computable in $r'$.
Hence, there is an $r'$-computable infinite tree $\mathcal{T}$ such that any path through $\mathcal{T}$ encodes a complete consistent type $\tau(x, \bar{g})$ extending $\tau''(x, \bar{g})$.  Since $S$ is a Scott set and  $\mathcal{T}$ is computable in  $r'\in S$, there is some $r\in S$ such that $r$ computes a complete extension $\tau(x, \bar{g})$ of $\tau''(x, \bar{g})$.

 Recall that ${G}'=\divhull{\bar{g}}$, and let $\Gamma'$ be the value set for $G'$.  We let
 \begin{eqnarray*}
  B=\{b\in G'\mid \tau(x, \bar{g})\vdash b\le x\} \\
  C=\{c\in G'\mid \tau(x, \bar{g})\vdash x\le c\}
  \end{eqnarray*}
  By quantifier elimination for divisible ordered abelian groups, to realize the type $\tau(x,\bar{g})$, it suffices to realize the partial type (also computable in $r\in S$)
  \begin{equation}\label{Gcut}
  \{b\le x\mid b\in B\}\cup    \{x\le c\mid c\in C\}.
  \end{equation}
If $\tau(x, \bar{g})\vdash x=b$ for any $b\in B$, then the type in (\ref{Gcut}) is realized by $b\in B\subset G$, and similarly for $x=c$ with $c\in C$, so suppose there are no such equalities.  Let $G''\succ G$ be such that there is some $x_0\in G''$ such that $G''\models \tau(x_0, \bar{g})$.  Consider the set $\Delta=\{v(d-x_0)\mid d\in\ G'\}$.  We examine three cases regarding the structure of $\Delta$.

\noindent
{\bf Case} 1 ({\em Immediate Transcendental}) - $\Delta$ has no largest element.

We observe that this case does not occur in our context as \mbox{$\divhull{\bar{g}, x_0}\supseteq G'$} has finite rank.   Thus, $\Delta$ is finite and has a maximum element.

 For the remaining two cases, we fix $d_0\in G'$ such that $v(d_0-x_0)$ is the maximum of $\Delta$.  We suppose that $d_0\in B$.  The argument in the case that $d_0\in C$ is symmetric.   Consider the partial type (which is also computable in $r\in S$)
  \begin{equation}\label{Gcut2}
  \{b-d_0< x'\mid b\in B\}\cup    \{x'< c-d_0\mid c\in C\}.
  \end{equation}
It is clear that if $x'$ satisfies this cut, then $x'+d_0$ satisfies (\ref{Gcut}).  We show that this cut is realized in $G$ in the remaining two cases.

\noindent
{\bf Case} 2 ({\em Residue Transcendental})  - $\Delta$ has a largest element, which is in $\Gamma'$

We can show as in the proof of Theorem \ref{doag} in \cite{sgr} the following claim.
\begin{claim}\label{restransGclaim} There exist $b_0\in B$ and $c_0\in C$ such that for all $b\in B$ and $c\in C$ with $b_0\le b$ and $c\le c_0$,
 \begin{eqnarray*}\label{RTcond}
v(b-d_0)=v(x_0-d_0)=v(c-d_0) \text{ and, hence, }\\
v(b-x_0)=v(x_0-d_0)=v(c-x_0).
\end{eqnarray*}
\end{claim}

%

By the claim, we have that for all $b\in B$ and $c\in C$ with  $b\ge b_0$ and $c\le c_0$,
\begin{eqnarray*}
\frac{b-d_0}{b_0-d_0}<\frac{x_0-d_0}{b_0-d_0}<\frac{c-d_0}{b_0-d_0}.
\end{eqnarray*}
Hence, the following partial type (also computable in $r\in S$)
  \begin{equation}\label{Gcut3}
  \{\frac{b-d_0}{b_0-d_0}< x'\mid b\in B \ \&\ b\ge b_0\}\cup    \{x'<\frac{c-d_0}{b_0-d_0}\mid c\in C\ \&\ c\le c_0\}
  \end{equation}
  is a cut of $\mathbb{R}$ filled  by some $\hat{r}\in\mathbb{R}$.  Since $r\in S$ computes the cut for  $\hat{r}$, we have that $\hat{r}\in S$.  Thus, by assumption (\ref{GSS}), there is some $\hat{g}\in G$ such that $\frac{\hat{g}}{b_0-d_0}=\hat{r}$ (since $\hat{r}\in S=A_{[b_0-d_0]}$).  Since $\hat{r}$ fills the cut described in (\ref{Gcut3}), by definition of $A_{[b_0-d_0]}$, the element $\hat{g}$ fills the cut described in (\ref{Gcut2}), as desired.

  \noindent
{\bf Case} 3 ({\em Group Transcendental}) - $\Delta$ has a largest element, which is not in $\Gamma'$


Consider the sets
\begin{eqnarray*}
\Delta_1=\{v(c-d_0)\mid c\in C\} \text { and } \Delta_2=\{v(b-d_0)\mid b\in B\ \&\ b>d_0\}.
\end{eqnarray*}
%
As in the proof of Theorem \ref{doag} in \cite{sgr}, we can show the following.
\begin{claim}\label{GrpTransClaim}

  $\Delta_1<v(d_0-x_0)<\Delta_2$.
 \end{claim}
%

Since $G'$ has finite rank, $\Delta_1$ and $\Delta_2$ are finite and form   a cut in the value set $\Gamma$. By  (\ref{GDLO}), $\Gamma$ is a dense linear ordering without endpoints, so there is some $y\in G$ with $y>0$ that fills this cut in $\Gamma$.  Then, for all $c\in C$ and $b\in B$ with $b> d_0$ we have  $v(b-d_0)>v(y)>v(c-d_0)$  so $b-d_0<y< c-d_0$.  Hence $y$ fills the cut given in (\ref{Gcut2}), finishing  the case where $\Delta$ has a largest element, which is not an element of $\Gamma'$.  This completes the proof that the properties stated are sufficient to guarantee that $G$ is recursively saturated.

%
%
%
\end{proof}

\section{Recursively saturated real closed fields}

\begin{dfn}
Let $r\in\mathbb{R}$.  Let $R$ be a real closed field.  Let $\bar{d}$ be a tuple of parameters in $R$.  We say a  length $\omega$ sequence  of elements \mbox{$(a_{i})_{ i<\omega} \subset \rc{\bar{d}}$}   is {\em computable in $r$}
if there is an $r$-computable sequence  of formulas $(\theta_{i}(x, \bar{d}))_{ i<\omega}$ such that $\theta_i(x, \bar{d})$ defines $a_i$ in $R$ for all $i<\omega$.

\end{dfn}
%

\begin{theo} \label{recsatrcf}
Let $R$ be a real closed field, $v$ its natural valuation, $G$ its value group and $k$ its residue field.
Then,  $R$ is recursively saturated in the language of ordered fields if and only if there is a Scott set $S$ such that

\begin{enumerate}[(i)]

\item\label{FArchComp} $G$ is recursively saturated with archimedean components all equal to $S$,

\item\label{Fresfield} $(k, +, \cdot, 0, 1, <)\cong (S, +, \cdot, 0, 1, <)$,

\item\label{Fpseudo} every pseudo Cauchy sequence of length $\omega$ in a subfield of $R$ of finite absolute transcendence degree over $\mathbb{Q}$ that is computable in an element of $S$ has a pseudo limit in $R$.

\item\label{FtypeScomp} every type realized by some $n$-tuple $\bar{a}$ in $R$ is computable in an element of $S$.
\end{enumerate}
\end{theo}
\begin{proof}
We first suppose that $R$ is recursively saturated.  We show that there is a Scott set $S$ such that conditions (\ref{FArchComp}), (\ref{Fresfield}),  (\ref{Fpseudo}), and (\ref{FtypeScomp})  hold with this $S$.

\noindent
(\ref{Fresfield}) Since $R$ is recursively saturated as an ordered  field, $(R, +, 0, <)$ is recursively saturated as a divisible ordered abelian group.  By Theorem \ref{recsatG}, there is some Scott set $S$ such that  the archimedean components $A_{[r], r}$ of $(R, +, 0, <)$  equal $S$ for all nonzero $r\in R$.   In particular, we have that $A_{[1], 1}=S$.  Hence, \mbox{$(k, +, 0, <)\cong (S, +, 0, <)$.}  Since $R$ is a real closed field,  $k$ is a real closed field as well.  Hence,  there is a subset $K\subset \mathbb{R}$ that is a real closed field isomorphic to $k$ and  an isomorphism $\phi$  from $(S, +, 0, <)$ to  $(K, +, 0, <)$.  By H\"{o}lder's Theorem, $\phi(x)=rx$ for some $r\in \mathbb{R}$.  We show  that there is a field isomorphism from $S$ to $K$.   Since $S$ is a Scott set, $S$ is a real closed subfield of $\mathbb{R}$.  Since $1\in S\cap K$, we have $r\in K$ and $\frac{1}{r}\in S$.   Since $S$ and $K$ are in fact sets of reals that form fields, $r, \frac{1}{r}\in S\cap K$.  Hence, $S=K$ (given $s\in S$, $\frac{s}{r}\in S$ so $\phi(\frac{s}{r})=s\in K$, and the other containment is similar).  So, the identity function from  $S$ to  $K$ is a {\em field} isomorphism, giving the desired result.

\noindent
(\ref{FtypeScomp}) Let $\gamma(\bar{x})$ be a type realized by the tuple $\bar{a}$ in $R$.  Let $r\in\mathbb{R}$ have the same Turing degree as $\gamma$, and let $\Psi$ be the Turing reduction computing $r$ from $\gamma$.  It suffices to  show $r\in S$.  Let $(\theta_i(\bar{x}))_{i\in\omega}$ be a fixed  effective enumeration of all formulas in the language of ordered fields.  Given $\sigma\in 2^{<\omega}$, let $\theta_\sigma(\bar{x})$ denote the conjunction of the formulas $\theta_i(\bar{x})$ such that $\sigma(i)=1$ and the formulas $\neg\theta_i(\bar{x})$ such that $\sigma(i)=0$.

We enumerate the formula
\begin{equation*}
\theta_{\sigma}(\bar{a})\rightarrow q< x< q'
\end{equation*}
into the partial type $\tilde{\gamma}(\bar{a}, x)$ if $\Psi$ computes that its output real must be between $q$ and $q'$ for $q< q'\in \mathbb{Q}$ from $\sigma$.  The partial type  $\tilde{\gamma}(\bar{a}, x)$ is computably enumerable and finitely satisfiable.  Since $R$ is recursively saturated, there is some $\tilde{r}$ so that  $\tilde{\gamma}(\bar{a}, \tilde{r})$ holds in $R$.  Since  $\gamma(\bar{a})$ holds in $R$ and $\Psi$ computes $r$ from $\gamma$, we have $r=\tilde{r}\in A_{[1], 1}=S$.

\noindent
(\ref{FArchComp})  We first show that all the archimedean components of $G$ equal $S$.
Let $r\in S$.
Since $S=A_{[1], 1}$ where  $A_{[1], 1}$ is an archimedean component of  $(R, +, 0, <)$, there is some $a\in R$ such that $r=\frac{a}{1}\in A_{[1], 1}$.  Let $g$ be a nonzero element of $G$ so $v(a_g)=g$ for some $a_g>0$ in $R$.   We show $r\in A_{[g], g} $.  Note that  $A_{[g], g}$ is  an archimedean component of  $(G, +, 0, <)$.  The group  $(G, +, 0, <)$ is isomorphic to a section of the {\em multiplicative} group $(R^{>0}, \cdot, 1, <)$ as opposed to the additive group $(R, +, 0, <)$.  If $r\in\mathbb{Q}$, then  $r \in A_{[g], g}$ since $G$ is divisible, so we may suppose $r\not\in \mathbb{Q}$.

Let $\delta(x,  a, a_g)$ be the partial type in the language of real closed fields consisting of all formulas with $q, q'\in\mathbb{Q}^{>0}$ and $q<q'$ of the form
\begin{equation*}
q<a< q'  \rightarrow a_g^q<x<a_g^{q'}
\end{equation*}

 The set of formulas $\delta(x, a, a_g)$ is computable and finitely satisfiable in ${R}$ since $R$ is real closed.  Since $R$ is recursively saturated, there is some $a'\in R$ so that  $\delta(a', a, a_g)$ holds in $R$.  Let $g'=v(a')\in G$.  Then, $\frac{g'}{g}=r$, so $r\in A_{[g'], g'}$.

Now, let $r\in A_{[g], g}$, so there is some $g'\in G$ such that $r=\frac{g'}{g}$.  Let $a_{g'}, a_g\in R$ be positive elements such that $v(a_{g'})=g'$ and $v(a_g)=g$.  If $r\in\mathbb{Q}$, then it is clear $r\in S$ as all rationals are computable.  Otherwise, let $\delta'(x,  a_{g'}, a_g)$ be the partial type in the language of real closed fields consisting of all formulas with $q, q'\in\mathbb{Q}^{>0}$ and $q<q'$ of the form
\begin{equation*}
a_g^q<a_{g'}<a_g^{q'}  \rightarrow q<x< q'
\end{equation*}
 As before, the set of formulas $\delta'(x,  a_{g'}, a_g)$ is computable and finitely satisfiable in ${R}$, so there is some $a\in R$ so that $\delta'(a,  a_{g'}, a_g)$ holds in $R$.   Then, $r=\frac{a}{1}\in k= S$.  Hence, we have $A_{[g], g}=S=k$ for all $g\in G$, so we can simply refer to $A_{[g]}$ instead of $A_{[g], g}$.

We now show that $G$ is recursively saturated.
Let $\bar{g}=(g_1, \ldots, g_n)$ be an $n$-tuple from $G$.   Let $\beta'(x, \bar{g})$ be a computable partial type in the language of  ordered abelian groups so that $\beta'(x, \bar{g})$ is finitely satisfiable in $G$.
 By the argument found at the beginning of the sufficiency proof for Theorem \ref{recsatG}, we can find an $r\in S$ such that $r$ computes a complete extension $\beta(x, \bar{g})$ of $\beta'(x, \bar{g})$ that is finitely satisfiable in $G$. Since $r\in S=k$, there exists some $a\in R$ such that $\frac{a}{1}=r$.

 Set ${G}'=\divhull{\bar{g}}$.  If $\beta\vdash x=g$ for any $g\in G'$, then we are done.  Otherwise,  let
 \begin{eqnarray*}
  B=\{b\in G'\mid \tau(x, \bar{g})\vdash b< x\} \\
  C=\{c\in G'\mid \tau(x, \bar{g})\vdash x< c\}
  \end{eqnarray*}
 As in Theorem \ref{recsatG}, it suffices to realize the partial type (also computable in $r\in S$) in $G$
  \begin{equation}\label{GcutRnec}
  \{b < x\mid b\in B\}\cup    \{x< c \mid c\in C\}
  \end{equation}
that describes a cut in $G$.  We translate realizing this cut in $G$ into realizing a particular partial type (in the language of ordered fields) in $R$.

 Let $R'=\rc{\bar{g}}$.  Take  $d_1, \ldots, d_n\in R'$ so that $d_1, \ldots, d_n>0$, $\{v(d_i)\mid 1\le i\le n\}$ is a basis for $G'$, and the multiplicative subgroup
 \begin{equation*}
 \{\prod_{i=1}^n d_i^{q_i}\in R'\mid q_i\in \mathbb{Q} \text{ for } 1\le i\le n\}
 \end{equation*}
  is a section for $G'$ in $R'$.  We show that there is a computably enumerable partial type in the language of ordered fields  $\tilde{\beta}(x, d_1, \ldots, d_n, a)$ (with parameters in $R$) that corresponds to  the  cut  described by $\beta(x, \bar{g})$ over $G'$.
    Note that
\begin{eqnarray*}
B=\big\{\sum_{i=1}^{n} q_iv(d_i)\in G'\mid \sum_{i=1}^{n}q_iv(d_i)<b\ \&\  b\in B\ \&\ q_1, \ldots, q_{n}\in \mathbb{Q} \big\} \text{ and }\\
C=\big\{\sum_{i=1}^{n}q_iv(d_i)\in G'\mid c<\sum_{i=1}^{n}q_iv(d_i)\ \&\ c\in C\ \&\ q_1,\ldots, q_{n}\in \mathbb{Q}\big\}.
\end{eqnarray*}

 Given some $(q_1, \ldots, q_n)\in \mathbb{Q}^n$, the statement that determines whether ${\sum_{i=1}^nq_iv(d_i)}$  is in $B$ or $C$ can be computably located in an effective listing of all formulas.   Since   $r\in S$ computes the complete type  $\beta(x, \bar{g})$, there is some Turing reduction $\Upsilon$ that computes  from $r$ whether a given \mbox{$(q_1, \ldots, q_n)\in \mathbb{Q}^n$} satisfies $\sum_{i=1}^nq_iv(d_i)\in B$ or $\sum_{i=1}^nq_iv(d_i)\in C$.

We now describe $\tilde{\beta}(x, d_1, \ldots, d_n, a)$.
For each  pair of rationals \mbox{$(q< q')$,} each stage $s\in\mathbb{N}$, and $(q_1, \ldots, q_n)\in \mathbb{Q}^n$, compute whether Turing reduction $\Upsilon$, using only the information that some real $\tilde{r}$ satisfies  $q<\tilde{r}< q'$, halts in $s$ steps and outputs whether $\sum_{i=1}^nq_iv(d_i)$ is in  $B$ or $C$.  If $\Upsilon$ halts in this situation, enumerate the  formula
\begin{eqnarray*}
q<a< q' \rightarrow \prod_{i=1}^nd_i^{q_i}< x \text{ if $\Upsilon$ computes that } \sum_{i=1}^nq_iv(d_i)\in C \text{ or }\\
q<a< q'\rightarrow  x <\prod_{i=1}^nd_i^{q_i}\text{ if $\Upsilon$ computes that } \sum_{i=1}^nq_iv(d_i)\in B
\end{eqnarray*}
into $\tilde{\beta}(x, d_1, \ldots, d_n, a)$.

 The partial type $\tilde{\beta}(x, d_1, \ldots, d_n, a)$ is finitely satisfiable in $R$ because $R$ is a dense linear ordering without endpoints.  Since $\tilde{\beta}(x, d_1, \ldots, d_n, a)$ is a computably enumerable and $R$ is recursively saturated, there exists some $d\in R$ so that  $\tilde{\beta}(d, d_1, \ldots, d_n, a)$ holds in $R$.  By our choice of $a$ and definition of $\tilde{\beta}$, we have that $B< v(d)<C$. So, $\beta(v(d), \bar{g})$ holds in $G$, as desired.

\noindent
(\ref{Fpseudo})
Let $(a_i)_{ i<\omega}\subset R'$  be a pseudo Cauchy sequence in a subfield $R'$ of $R$ of finite absolute transcendence degree over $\mathbb{Q}$.  Moreover, suppose that $(a_i)_{ i<\omega}$ is computable in $r\in S=k$. By definition, there is an $r$-computable sequence  of formulas $(\theta_{i}(x, \bar{y}))_{ i<\omega}$ and tuple $\bar{d}$ from $R'$ such that $\theta_i(x, \bar{d})$ defines $a_i$ in $R$, where  $R'=\rc{\bar{d}}$.  Let $\Phi$ denote the Turing reduction from $r$ to this sequence.   Since $r\in S=k$, there exists some $a\in R$ such that $\frac{a}{1}=r$.

 For each  pair of rationals $q< q'$ and  any $i, s\in\mathbb{N}$, compute whether Turing reduction $\Phi$, using only the information that some real $\tilde{r}$ satisfies  $q<\tilde{r}< q'$, halts in $s$ steps and outputs  indices for  formulas $\theta_i(x, \bar{y})$ and  $\theta_{i+1}(x, \bar{y}))$.   If $\Phi$ halts in this situation, enumerate the following formula into the partial type $\kappa(x, \bar{d}, a)$
 \begin{equation*}
q< a<q'\rightarrow 
[(\exists z_{i}, z_{i+1})(\theta_i(z_i, \bar{d}) \, \wedge\, \theta_{i+1}(z_{i+1}, \bar{d})) \  \wedge\  n|x- z_{i+1}|<|z_i- z_{i+1}|]
\end{equation*}

\noindent for each $n\in\omega$.  Note that $\kappa(x, \bar{d}, a)$ is computably enumerable and  finitely satisfiable in $R$.  Given a finite set of formulas $D\subset \kappa(x, \bar{d}, a)$, let $j<\omega$ be the largest number  such that  $\theta_j(x, \bar{d})$ appears as a subformula of an element in $D$.  Then, $a_{j+1}\in R$ satisfies all formulas in $D$ since $(a_i)_{ i<\omega}$ is pseudo Cauchy.  Since $R$ is  recursively saturated, there is some $\tilde{a}\in R$ such that
$\kappa(\tilde{a}, \bar{d}, a)$ holds in $R$.
This implies that $v(\tilde{a}-a_{i+1})>v(a_{i+1}-a_{i})$ for all $i<\omega$.  From \mbox{$v(\tilde{a}-a_i)\ge\min\{v(\tilde{a}-a_{i+1}), v(a_{i+1}-a_i)\}$} it follows  that  \mbox{$v(\tilde{a}-a_{i})= v(a_{i+1}-a_i)$} for all $i<\omega$.  Hence, $\tilde{a}$ is a pseudo limit of $(a_{i})_{ i < \omega}$, as required, and the four   conditions (\ref{FArchComp}), (\ref{Fresfield}),  (\ref{Fpseudo}), and (\ref{FtypeScomp}) are necessary if $R$ is recursively saturated.

Let $R$ be a real closed field.   We assume that there is a Scott set $S$ for which conditions (\ref{FArchComp}), (\ref{Fresfield}),  (\ref{Fpseudo}), and (\ref{FtypeScomp})  hold for $R$ and $S$.  We  show that $R$ is recursively saturated.

Let $\bar{a}=(a_1, \ldots, a_n)$ be a finite tuple from $R$, and let $\tau'(x, \bar{a})$ be a computable set of formulas that is finitely satisfiable in $R$.    We first extend $\tau'(x, \bar{y})$ to a complete type $\tau(x, \bar{y})$ so that $\tau(x, \bar{a})$ is also finitely satisfiable in $G$.
We first make an intermediate extension $\tau''(x, \bar{y})$ of $\tau'(x, \bar{y})$.  Let $\gamma(\bar{y})$ be the complete type of $\bar{a}$ in $R$.  We then let \mbox{$\tau''(x, \bar{y})=\tau'(x, \bar{y})\cup \gamma(\bar{y})$.}  By condition (\ref{FtypeScomp}), the type $\gamma(\bar{y})$ is computable in some  $r'\in S$, so $\tau''(x, \bar{y})$ is as well.  Moreover, $\tau''(x, \bar{a})$ is finitely satisfiable in $R$.  Hence, there is an $r'$-computable infinite tree $\mathcal{T}$ such that any path through $\mathcal{T}$ encodes a complete consistent type  extending $\tau''(x, \bar{a})$.  Since $S$ is a Scott set and  $\mathcal{T}$ is computable in  $r'\in S$, there is some $r\in S$ such that $r$ computes a complete extension $\tau(x, \bar{a})$ of $\tau''(x, \bar{a})$.

 Set ${R}'=\rc{\bar{a}}$.  We set
 \begin{eqnarray*}
  B=\{b\in R'\mid \tau(x, \bar{a})\vdash b\le x\} \text{ and}\\
  C=\{c\in R'\mid \tau(x, \bar{a})\vdash x\le c\}.
  \end{eqnarray*}
Real closed fields, like divisible ordered ablian groups, have  quantifier elimination. Hence, to realize the type $\tau(x, \bar{a})$, it suffices to realize the partial type (also computable in $r\in S$)
  \begin{equation}\label{Fcut}
  \{b\le x\mid b\in B\}\cup    \{x\le c\mid c\in C\}.
  \end{equation}
If $\tau(x, \bar{a})\vdash x=b$ for any $b\in B$, then the type in (\ref{Fcut}) is realized by $b\in B\subset R$, and similarly for $x=c$ with $c\in C$, so suppose there are no such equalities.  Let $R''\succ R$ be such that there is some $x_0\in R''$ satisfying $R''\models \tau(x_0, \bar{a})$.  Consider the set $\Delta=\{v(d-x_0)\mid d\in R'\}$.  We examine  three cases  for the structure of $\Delta$, as we did in Theorem \ref{recsatG} in the group case.

\noindent
{\bf Case} 1 ({\em Immediate Transcendental}) - $\Delta$ has no largest element.
In this case,  for all $d\in R'$ there is a $d'\in R'$ such that   \mbox{$v(d-x_0)<v(d'-x_0)$.}  We construct a  pseudo Cauchy sequence $(a_i)_{i<\omega}$ that is  computable in some element of $S$ and has a pseudo limit $a\in R$ satisfying $B<a<C$.
  By effective quantifier elimination for real closed fields, there is a computable enumeration of formulas $\{\psi_i(x, \bar{a})\}_{i<\omega}$ such that
   \begin{enumerate}[(a)]
  \item every element in $R'$ is defined by exactly one formula in this sequence and
  \item if $a_i$ and $a_j$ are defined by $\psi_i(x, \bar{a})$ and $\psi_j(x, \bar{a})$ respectively, then determining whether $a_i< a_j$ and whether $a_i\in B$ or $a_i\in C$ in $R'$ is $r$-computable.
  \end{enumerate}
%
    Let $a_i$ denote the element in $R'$  that satisfies the definition $\psi_i(a_i, \bar{a})$.
   We define a tree $\mathcal{T}\subset 2^{<\omega}$ computable in $r$.  For any $\sigma\in 2^{<\omega}$, we put $\sigma\in \mathcal{T}$ if the following two properties hold.

   \begin{enumerate}[(I)]
   \item\label{Ftreedef1} For all  $i<length(\sigma)$, set $a'$ equal to $0$ if for all $j\le i$, $\sigma(j)=0$, and otherwise, set $a'$ equal to $a_{j'}$ where \mbox{$j'=\max \{j\le i\mid\sigma(j)=1\}$}.  Then,
  \begin{equation*}
  (\forall j\le i)(a_{j}\in B\implies a_{j}\le a') \ \&\   (a_{j}\in C\implies a'\le a_{j}))
  \end{equation*}
   \item\label{Ftreedef2} $(\forall\, i< j<k<n=length(\sigma))\\ (\sigma(i)=\sigma(j)=\sigma(k)=1\implies n|a_k-a_j|<|a_j-a_i|)$
   \end{enumerate}
%
It is clear $\mathcal{T}$ is a tree by definition.  We now show $\mathcal{T}$ is infinite.      Since $\Delta$ has no largest element, there exists a cofinal sequence in $\Delta$.  Moreover, since $R'$ is countable and $B<x_0<C$ in $R''$, we can take this cofinal sequence to have the form $(v(a_{i_l}-x_0))_{l<\omega}$ and to satisfy the following two properties.
  \begin{enumerate}[(a)]
 \item The sequences $(i_l)_{l<\omega}$ and $(v(a_{i_l}-x_0))_{l<\omega}$ are  increasing.
\item For each $n<\omega$,  if we set $a'$  equal to $0$ if no $j\le n$ equals some $i_l$ and we set $a'$ equal to $a_{i_{l'}}$ where index \mbox{$i_{l'}=\max \{i_l\le n\}$} otherwise,  then \\
$(\forall\, j\le n)(a_{j}\in B\implies a_{j}\le a') \ \&\   (a_{j}\in C\implies a'\le a_{j}))$
\end{enumerate}

Let  $\mathcal{P}'\in 2^\omega$ be defined so that $\mathcal{P}'(j)=1$ if and only if $j=i_l$ for some $l\in\omega$.  We show that $\mathcal{P}'$ is a path through $\mathcal{T}$, so $\mathcal{T}$ is infinite.  Let $\sigma_n=\mathcal{P}'\upharpoonright n$.  It is clear that $\sigma_n$ satisfies (\ref{Ftreedef1}) by definition.  We show that $\sigma_n$ satisfies (\ref{Ftreedef2}).  Suppose $i<j<k<n$ with
\begin{equation*}
\sigma_n(i)=\sigma_n(j)=\sigma_n(k)=1,
\end{equation*}
 i.e., $i=i_{l}$, $j=i_{l'}$ and $k=i_{l''}$ with $l<l'<l''$.  It suffices to show that $v(a_j-a_i)<v(a_k-a_j)$.  We have that
\begin{equation*} v(a_{i_l}-x_0)<v(a_{i_{l'}}-x_0)<v(a_{i_{l''}}-x_0)
\end{equation*}
and
\begin{eqnarray*}
v(a_j-a_i)=\min(v(a_{i_{l'}}-x_0), v(a_{i_{l}}-x_0))=v(a_{i_l}-x_0)\\
v(a_k-a_j)=\min(v(a_{i_{l''}}-x_0), v(a_{i_{l'}}-x_0))=v(a_{i_{l'}}-x_0) \text{ so }\\
v(a_j-a_i)<v(a_k-a_j) \text{ as desired.}
 \end{eqnarray*}
 Hence, $\mathcal{T}$ is an infinite tree computable in $r$.  Since $S$ is a Scott set, there exists a path $\mathcal{P}$ through $\mathcal{T}$  computable in some $t\in S$.  Since $B$ and $C$ form a proper cut in $R'$,  there are infinitely many $j<\omega$ such that \mbox{$\mathcal{P}(j)=1$} by property (\ref{Ftreedef1}) of $\mathcal{T}$.
We then can compute in $t$,  for each $l<\omega$, the index $k_l$ such that \mbox{$\mathcal{P}(k_l)=1$} and $|\{j\le k_l\mid \mathcal{P}(j)=1\}|=k_l$.
By property (\ref{Ftreedef2}) of the definition of $\mathcal{T}$, the sequence $(a_{k_l})_{l<\omega}$ is  pseudo Cauchy.
Since    $(\psi_{k_l}(x, \bar{y}))_{l<\omega}$ is computable in $t$, the sequence $(a_{k_l})_{l<\omega}$ (defined by this sequence of formulas over $\bar{a}$) has a pseudo limit $a\in R$ by assumption (\ref{Fpseudo}).

We show that $B<a<C$ holds in $R$, and so $a$ realizes the type in (\ref{Fcut}).
Let $b\in B\subset R'$.  We claim that $b<a$.  Otherwise, $a\le b <x_0$.  By definition of $\mathcal{T}$, there exists some $l<\omega$ such that $a\le b\le a_{k_l}< a_{k_{l+1}}$.  Then, $v(a- a_{k_{l+1}})\le v(a-a_{k_l})$.  Since $a$ is a pseudo limit for $(a_{k_l})_{l<\omega}$,
\begin{equation*}
v(a- a_{k_{l+1}})=v(a_{k_{l+2}}-a_{k_{l+1}})> v(a_{k_{l+1}}-a_{k_{l}})=v(a- a_{k_{l}}),
\end{equation*}
 a contradiction, so we have shown $b<a$.  The argument that $a<c$  for any $c\in C$ is similar.

 \noindent
{\bf Case} 2 ({\em Residue  Transcendental}) - $\Delta$ has a largest element $g\in v(R')$

Assume that $\Delta$ has a largest element $g\in v(R')$.  Let $a>0$ be such that $a, d_0\in R'$ and $v(d_0-x_0)=g=v(a)$.

\begin{claim} There exists $b_0\in B$ and $c_0\in C$ such that for all $b\in B$ with $b\ge b_0$ and for all $c\in C$ with $c\le c_0$, we have
\begin{eqnarray*}
v(b-d_0)=g=v(a)=v(c-d_0) \text{ and, hence,}\\
v(b-x_0)=g=v(a)=v(x_0-d_0)=v(c-x_0).
\end{eqnarray*}
\end{claim}

Like the corresponding Claim  \ref{restransGclaim}, its proof is a straightforward adaptation of the proof of the analogous statement in Theorem \ref{doag} in \cite{sgr}.

%

Consider the partial type (also computable in $r$):
\begin{equation}\label{Fcutresfield}
\Big\{\frac{b-d_0}{a}<x\mid b\in B\ \&\ b\ge b_0\Big\}\cup \Big\{x<\frac{c-d_0}{a}\mid c\in C\ \&\ c\le c_0\Big\}
\end{equation}

If some $x'$ realizes the type in (\ref{Fcutresfield}) then $x=a\cdot x'+d_0$ realizes the type in (\ref{Fcut}).  So, it suffices to find such an $x'\in R$.  Again, suppose $d_0\in B$, as the case where $d_0\in C$ is symmetric.

By the claim, for all $b\in B$ and $c\in C$ with $b_0\le b<c\le c_0$
\begin{equation*}
v\Big(\frac{b-d_0}{a}\Big)=v\Big(\frac{x_0-d_0}{a}\Big)=v\Big(\frac{c-d_0}{a}\Big)=0
\end{equation*}
Furthermore, if we take the residues of these elements, we have
\begin{equation*}
\overline{\frac{b-d_0}{a}}<\overline{\frac{x_0-d_0}{a}}<\overline{\frac{c-d_0}{a}}.
\end{equation*}
All inequalities in the line above are strict since otherwise $g=v(x_0-d_0)$ is not the maximum of $\Delta$.   Hence, the two sets
\begin{eqnarray*}
\{q\in\mathbb{Q}\mid q<\frac{b-d_0}{a}\ \&\ b\in B\ \&\ b\ge b_0\} \\ \{q\in\mathbb{Q}\mid  \frac{c-d_0}{a}<q\ \&\ c\in C\ \&\ c\le c_0\}
\end{eqnarray*}
form a cut in $\mathcal{R}$ that is computable in $r$.  Let $r'\in \mathbb{R}$ fill this cut.  Since $r'$ is computable in $r$, we have $r'\in S\cong k$ by assumption (\ref{Fresfield}).  Thus, there is some $x'\in R$  that realizes the partial type in (\ref{Fcutresfield}), as desired.

 \noindent
{\bf Case} 3 ({\em  Group Transcendental}) - $\Delta$ has a largest element $g\not\in v(R')$

Let $d_0\in R'$ such that $v(d_0-x_0)=g$ is the maximum of $\Delta$.  We suppose that $d_0\in B$; the case that $d_0\in C$ is similar.

  Consider the sets
\begin{eqnarray*}
\Delta_1=\{v(c-d_0)\mid c\in C\} \text { and } \Delta_2=\{v(b-d_0)\mid b\in B\ \&\ b>d_0\}.
\end{eqnarray*}

\begin{claim}
$\Delta_1<g<\Delta_2$.
 \end{claim}

 As for the corresponding Claim \ref{GrpTransClaim} in the group case, the proof of the above claim can be found in  Theorem \ref{doag} in \cite{sgr}.
%

By the claim,
\begin{equation*}
\eta(y)=\{v(c-d_0)<y\mid c\in C\}\ \cup \ \{y<v(b-d_0)\mid b\in B\ \&\ b>d_0\}
\end{equation*}
is a type in $G$ with parameters in $G'=v(R')$ that describes a cut in $G$.    The real closed field $R'$ has finite absolute transcendence degree, so $G'$ has finite rational rank (see \cite{sz}, Section 10).
%
Take  $d_1, \ldots, d_n\in R'$ so that $\{v(d_i)\mid 1\le i\le n\}$ is a basis for $G'$ and the multiplicative subgroup
\begin{equation*}
\{\prod_{i=1}^n d_i^{q_i}\in R'\mid q_i\in \mathbb{Q} \text{ for } 1\le i\le n\}
\end{equation*}
 is a section for $G'$ in $R'$.  We show that there is a computably enumerable partial type   $\tilde{\eta}(y, v(d_1), \ldots, v(d_n), h)$ (with parameters in $G$) that describes the same cut  over $G'$ as $\eta(y)$.

    Note that
\begin{eqnarray*}
\Delta_1=\big\{\sum_{i=1}^nq_iv(d_i)\in G'\mid \prod_{i=1}^nd_i^{q_i}<c-d_0\ \&\  c\in C\ \&\ q_i\in \mathbb{Q} \big\} \text{ and }\\
\Delta_2=\big\{\sum_{i=1}^nq_iv(d_i)\in G'\mid b-d_0< \prod_{i=1}^nd_i^{q_i}\ \&\ b\in B\ \&\ b>d_0\ \&\ q_i\in \mathbb{Q}\big\}.
\end{eqnarray*}
Recall that  $r\in S$ computes the complete type  $\tau(x, \bar{a})$ extending the computable partial type $\tau'(x, \bar{a})$ we wish to realize in $R$.   Moreover, given some $(q_1, \ldots, q_n)\in \mathbb{Q}^n$, the statement that determines whether $\sum_{i=1}^nq_iv(d_i)$  is in $\Delta_1$ or $\Delta_2$ can be computably located in $\tau$.   Hence, there is some Turing reduction $\Upsilon$ that computes  from $r$ whether a given \mbox{$(q_1, \ldots, q_n)\in \mathbb{Q}^n$} satisfies $\sum_{i=1}^nq_iv(d_i)\in \Delta_1$ or $\sum_{i=1}^nq_iv(d_i)\in\Delta_2$.

Take a nonzero $g\in v(R)$; such a $g$ exists by (\ref{FArchComp}) and Theorem \ref{recsatG}.  Since $r\in S=A_{[g], g}$ by assumption (\ref{FArchComp}), there exists some $g_r\in v(R)$ such that $\frac{g_r}{g}=r$.

We now describe $\tilde{\eta}(y, v(d_1), \ldots, v(d_n), g_r)$.
For each  pair of rationals $(q< q')$, each stage $s\in\mathbb{N}$, and $(q_1, \ldots, q_n)\in \mathbb{Q}^n$, compute whether Turing reduction $\Upsilon$, using only the information that some real $\tilde{r}$ satisfies  $q<\tilde{r}< q'$, halts in $s$ steps and outputs whether $\sum_{i=1}^nq_iv(d_i)$ is in  $\Delta_1$ or $\Delta_2$.  If $\Upsilon$ halts in this situation, enumerate either the formula
\begin{eqnarray*}
qg<g_r< q'g \rightarrow \sum_{i=1}^nq_iv(d_i)< y \text{ if computation says } \sum_{i=1}^nq_iv(d_i)\in\Delta_1\\
\text{ or the formula }\\
qg<g_r< q'g \rightarrow  y <\sum_{i=1}^nq_iv(d_i)\text{ if computation says } \sum_{i=1}^nq_iv(d_i)\in\Delta_2
\end{eqnarray*}
into $\tilde{\eta}(y, v(d_1), \ldots, v(d_n), g_r)$.  The partial type $\tilde{\eta}(y, v(d_1), \ldots, v(d_n), g_r)$ is  computably enumerable, and it is finitely satisfiable in $G$ because $G$ is a divisible group.  By assumption (\ref{FArchComp}), $G$ is recursively saturated and this implies  there exists some $h\in G$ so that  $\tilde{\eta}(h, v(d_1), \ldots, v(d_n), g_r)$ holds in $G$.  By our choice of $g_r$ and definition of $\tilde{\eta}$, we have that $\Delta_1< h<\Delta_2$.

Let $a_h\in R$ satisfy $v(a_h)=h$ and $a>0$.   Then,
\begin{eqnarray*}
(\forall c\in C)(\forall b\in B)(b>d_0\implies v(c-d_0)<v(a)<v(b-d_0) \\{\text { and, hence, }}
(\forall c\in C)(\forall b\in B)(b>d_0\implies b-d_0<a<c-d_0).
\end{eqnarray*}
Thus, $B<a+d_0<C$, so $a+d_0$ realizes the type given in (\ref{Fcut}), as required.

Therefore, in each of the three cases, we satisfied the type $\tau$.  Hence, $R$ is recursively saturated.  This completes the sufficiency direction of the proof.
\end{proof}

It is unclear whether condition (\ref{FtypeScomp}) follows from the other three conditions listed in Theorem \ref{recsatrcf}.  In Theorem \ref{recsatG}, we used a valuation basis for $G$ to avoid the need for such a condition.  However,  a real closed field of finite absolute transcendence degree need not admit a valuation transcendency basis; see \cite{kuh} for a precise definition of valuation transcendency basis and counterexamples.

\medskip
\noindent {\bf Acknowledgement}: \ The authors  thank Victor Harnik
for making \cite{hr} available to us. Research visits of the first
and second authors were partially supported by funds from the
Gleichstellungsrat of the University of Konstanz. The third author
was partially supported by NSF DMS-1100604. This research was
partially done while the third author was a visiting fellow at the
Isaac Newton Institute for the Mathematical Sciences in the
`Semantics \& Syntax' program.

\vs

\end{document}